\documentclass[12pt]{amsart}
\usepackage{amsmath,amssymb,amsbsy,amsfonts,latexsym,amsopn,amstext,
	amsxtra,euscript,amscd,bm}
\usepackage{url}
\usepackage[colorlinks,linkcolor=blue,anchorcolor=blue,citecolor=blue,backref=page]{hyperref}
\usepackage{color}
\usepackage{float}
\hypersetup{breaklinks=true}
\usepackage{multirow}
\usepackage[norefs,nocites]{refcheck}

\usepackage{url}
\usepackage[colorlinks,linkcolor=blue,anchorcolor=blue,citecolor=blue,backref=page]{hyperref}
\usepackage{color}
\usepackage{mathtools}
\usepackage{todonotes}

\newcommand{\diag}{\operatorname{diag}}

\begin{document}
	
	\newtheorem{theorem}{Theorem}
	\newtheorem{lemma}[theorem]{Lemma}
	\newtheorem{claim}[theorem]{Claim}
	\newtheorem{cor}[theorem]{Corollary}
	\newtheorem{prop}[theorem]{Proposition}
	\newtheorem{example}[theorem]{Example}
	\newtheorem{definition}{Definition}
	\newtheorem{quest}[theorem]{Question}
	\newcommand{\hh}{{{\mathrm h}}}
	
	\numberwithin{equation}{section}
	\numberwithin{theorem}{section}
	\numberwithin{table}{section}
	
	\def\sssum{\mathop{\sum\!\sum\!\sum}}
	\def\ssum{\mathop{\sum\ldots \sum}}
	\def\iint{\mathop{\int\ldots \int}}

	\def\squareforqed{\hbox{\rlap{$\sqcap$}$\sqcup$}}
	\def\qed{\ifmmode\squareforqed\else{\unskip\nobreak\hfil
			\penalty50\hskip1em\null\nobreak\hfil\squareforqed
			\parfillskip=0pt\finalhyphendemerits=0\endgraf}\fi}
	
	\newfont{\teneufm}{eufm10}
	\newfont{\seveneufm}{eufm7}
	\newfont{\fiveeufm}{eufm5}
	%
	%
	\newfam\eufmfam
	\textfont\eufmfam=\teneufm
	\scriptfont\eufmfam=\seveneufm
	\scriptscriptfont\eufmfam=\fiveeufm
	%
	%
	\def\frak#1{{\fam\eufmfam\relax#1}}
	%

	
	\newcommand{\bflambda}{{\boldsymbol{\lambda}}}
	\newcommand{\bfmu}{{\boldsymbol{\mu}}}
	\newcommand{\bfxi}{{\boldsymbol{\xi}}}
	\newcommand{\bfrho}{{\boldsymbol{\rho}}}
	
	\def\fK{\mathfrak K}
	\def\fT{\mathfrak{T}}

	\def\fA{{\mathfrak A}}
	\def\fB{{\mathfrak B}}
	\def\fC{{\mathfrak C}}

	\def\vec#1{\mathbf{#1}}

	\def\squareforqed{\hbox{\rlap{$\sqcap$}$\sqcup$}}
	\def\qed{\ifmmode\squareforqed\else{\unskip\nobreak\hfil
			\penalty50\hskip1em\null\nobreak\hfil\squareforqed
			\parfillskip=0pt\finalhyphendemerits=0\endgraf}\fi}

	\def\cA{{\mathcal A}}
	\def\cB{{\mathcal B}}
	\def\cC{{\mathcal C}}
	\def\cD{{\mathcal D}}
	\def\cE{{\mathcal E}}
	\def\cF{{\mathcal F}}
	\def\cG{{\mathcal G}}
	\def\cH{{\mathcal H}}
	\def\cI{{\mathcal I}}
	\def\cJ{{\mathcal J}}
	\def\cK{{\mathcal K}}
	\def\cL{{\mathcal L}}
	\def\cM{{\mathcal M}}
	\def\cN{{\mathcal N}}
	\def\cO{{\mathcal O}}
	\def\cP{{\mathcal P}}
	\def\cQ{{\mathcal Q}}
	\def\cR{{\mathcal R}}
	\def\cS{{\mathcal S}}
	\def\cT{{\mathcal T}}
	\def\cU{{\mathcal U}}
	\def\cV{{\mathcal V}}
	\def\cW{{\mathcal W}}
	\def\cX{{\mathcal X}}
	\def\cY{{\mathcal Y}}
	\def\cZ{{\mathcal Z}}
	\newcommand{\rmod}[1]{\: \mbox{mod} \: #1}

	\def\vr{\mathbf r}
	
	\def\e{{\mathbf{\,e}}}
	\def\ep{{\mathbf{\,e}}_p}
	\def\em{{\mathbf{\,e}}_m}
	
	\def\Tr{{\mathrm{Tr}}}
	\def\Nm{{\mathrm{Nm}}}
	\def\rM{{\mathrm{M}}}
	
	\def\SL{{\mathrm{SL}}}
	
	\def\SS{{\mathbf{S}}}
	
	\def\lcm{{\mathrm{lcm}}}

	\def\({\left(}
	\def\){\right)}
	\def\fl#1{\left\lfloor#1\right\rfloor}
	\def\rf#1{\left\lceil#1\right\rceil}
	
	\def\eps{\varepsilon}
	\def\al{\alpha}
	\def\be{\beta}
	\def\N{\mathbb{N}}
	\def\L{\mathbb{L}}
	
	\def\mand{\qquad \mbox{and} \qquad}
	\def\mor{\qquad \mbox{or} \qquad}
	
	\def\ccr#1{\textcolor{red}{#1}}
	\def\cco#1{\textcolor{orange}{#1}}
	\def\ccc#1{\textcolor{cyan}{#1}}

	\newcommand{\commA}[2][]{\todo[#1,color=green!60]{A: #2}}
	\newcommand{\commI}[2][]{\todo[#1,color=yellow]{I: #2}}
	\newcommand{\commM}[2][]{\todo[#1,color=red]{M: #2}}
	%
	%
	%
	%



	\hyphenation{re-pub-lished}

	\parskip 4pt plus 2pt minus 2pt

	\mathsurround=1pt

	\def\bfdefault{b}
	\overfullrule=5pt

	\def \F{{\mathbb F}}
	\def \K{{\mathbb K}}
	\def \Z{{\mathbb Z}}
	\def \Q{{\mathbb Q}}
	\def \R{{\mathbb R}}
	\def \C{{\mathbb C}}
	\def\Fp{\F_p}
	\def \fp{\Fp^*}
	
	\def\Kmn{\cK_p(m,n)}
	\def\psmn{\psi_p(m,n)}
	\def\AI{\cA_p(\cI)}
	\def\BIJ{\cB_p(\cI,\cJ)}
	\def \xbar{\overline x_p}

	\title[Euclidean Minima  of Algebraic Number Fields]{Euclidean Minima  of Algebraic Number Fields}

	\author{Art\= uras Dubickas}
	\address{Institute of Mathematics, Faculty of Mathematics and Informatics, Vilnius University, Naugarduko 24,
		LT-03225 Vilnius, Lithuania}
	\email{arturas.dubickas@mif.vu.lt}

	\author{Min Sha}
	\address{School of Mathematical Sciences, South China Normal University, Guangzhou, 510631, China}
	\email{min.sha@m.scnu.edu.cn}

	\author{Igor E. Shparlinski}
	\address{School of Mathematics and Statistics, University of New South Wales,
		Sydney, NSW 2052, Australia}
	\email{igor.shparlinski@unsw.edu.au}

	\begin{abstract} 
		In this paper, we use  some of our previous results to improve an upper bound of 
		Bayer-Fluckiger, Borello and Jossen on the Euclidean minima  of algebraic number fields. 
		Our bound depends on the degree $n$ of the field, its signature, discriminant and the Hermite constant in dimension $n$. 
		
	\end{abstract}
	

	\maketitle

	\section{Introduction}
	
	Let $\K$ be a number field
	of degree $n$ over the field of rational numbers $\Q$, and let 
	$\Z_{\K}$ be the ring of integers of $\K$. We use $\Nm_{\K/\Q}$ to denote the {\it norm map\/}
	$\Nm_{\K/\Q}: \K \to \Q$ (which maps $\gamma \in \K$ to  $\prod_{i=1}^n |\sigma_i(\gamma)|$, where $\sigma_1,\dots,\sigma_n$ are all $n$ embeddings of $\K$ into $\C$), and define the parameter 
	$$
	\rM(\K) = \sup_{\alpha\in \K} \inf_{\beta \in \Z_\K} |\Nm_{\K/\Q}(\alpha-\beta)|.
	$$
	The latter is known as the {\it Euclidean minimum\/} of $\K$ and 
	is  a classical  object of study; we refer to~\cite{B-F,B-FBJ,B-FM1,B-FM,el,Lez,mcg,ShWa} for an overview of some previous results on $\rM(\K)$.
	In particular,  as we note in Section~\ref{sec:Latt}, 
	$\rM(\K)$ is an inhomogeneous minimum in a lattice which is naturally associated with 
	the field $\K$. 
	
By an ingenious application of an idea of McMullen~\cite{McM}, 
	Bayer-Fluckiger, Borello and Jossen~\cite{B-FBJ} have obtained a new parametric bound on the Euclidean minimum 
	$\rM(\K)$.\footnote{See the latest 2023 version of this paper at {\tt arXiv:1511.00908v3}, where some corrections in the proof of~\cite[Theorem~5.3]{B-FBJ} have been made.}
	To present their result we need to introduce some notation. 
	
	We say that the field $\K$ of degree $n$ over $\mathbb Q$ is of {\it signature\/} $(r,s)$ 
	if it has $r$ real embeddings  and $s$   pairs of complex conjugate embeddings, so that 
	\begin{equation}\label{nrs}
		n=[\K:\Q]=r+2s.
	\end{equation}
	We also denote by $D_{\K}$ the absolute value of the  discriminant of the field $\K$.
	Finally, let $\gamma_n$ be the  {\it Hermite constant\/} in dimension $n$; 
	see Section~\ref{sec:Latt} for a precise definition. 
	
	Then, by~\cite[Theorem~5.3]{B-FBJ}, for any positive  integer $a \le r+s$ 
	one has
	$$
	\rM(\K) \le 2^{\frac{-s(s+a)}{a}}\big(2^{s-a} \gamma_n^{s+a} n^{-s}\big)^{\frac{n}{2a}} D_{\K}^{\frac{s+a}{2a}},
	$$
	which, by~\eqref{nrs}, can be equivalently written as 
	\begin{equation}\label{bflu000}
		\rM(\K)\le 2^{-n + \frac{ar + rs}{2a}}n^{\frac{-ns}{2a}} \gamma_n^{\frac{n(s+a)}{2a}} D_{\K}^{\frac{s+a}{2a}}.
	\end{equation}
	
	In this paper, we use some results and ideas from~\cite{DSS} (see also two subsequent papers~\cite{AD1, AD2}
	and well as its predecessors~\cite{RoTs,Tsfas}) to  improve~\eqref{bflu000}
	by the factor $2^{\frac{ar+rs}{2a}}$. 
	
	\begin{theorem}
		\label{thm: MKa}
		For any positive  integer $a \le r+s$ 
		we have
		$$
		\rM(\K)\le  2^{-n} n^{\frac{-ns}{2a}} \gamma_n^{\frac{n(s+a)}{2a}} D_{\K}^{\frac{s+a}{2a}}.
		$$
	\end{theorem}
	
	In particular, for fields with large discriminant $D_{\K}$,  
	when the choice $a = r + s$ is optimal, 
	this leads to the following upper bound:
	
	\begin{cor}
		\label{cor: MK}
		We have
		$$
		\rM(\K) \le  2^{-n} n^{\frac{-ns}{2(r+s)}} \gamma_n^{\frac{n^2}{2(r+s)}} D_{\K}^{\frac{n}{2(r+s)}}.
		$$
	\end{cor}
	
	Theorem~\ref{thm: MKa} and Corollary~\ref{cor: MK} can now be combined with some recent  upper bounds and explicit values 
	on the Hermite constants to yield more explicit bounds for the Euclidean minimum $\rM(\K)$. 
	For example, for any integer $n \ge 1$, we have
	$$
	\gamma_n < \frac{n}{8} + \frac{6}{5}
	$$ 
	by~\cite[Theorem~1]{Wen1}, 
	and  
	$$
	\gamma_n < \frac{2n}{17} + 2
	$$ 
	by~\cite[Theorem~1]{Wen2}.
	
	In particular, for small  $n$, in Section~\ref{sec:calc}
	we present some numerical improvements of the results from the table of~\cite[Section~5.4]{B-FBJ}
	(these bounds are all based on the precise knowledge of $\gamma_n$ for small $n$).  
	In Section~\ref{sec:calc} we also confirm that the bound in Theorem~\ref{thm: MKa} is better than 
	a general bound for $\rM(\K)$ in~\cite{B-F} for small $n$. 
	
	In Section~2, we shall present all results
	necessary for the proof of Theorem~\ref{thm: MKa}, and then complete the proof in Section~3.

	\section{Homogeneous and inhomogeneous minima in lattices}
	\label{sec:Latt}
	
	We first recall some standard notions and results related to lattices from algebraic number fields. 
	For a number field $\K$  of degree $n$ and of signature $(r,s)$, 
	we assume that its 
	real embeddings are $\sigma_1,\ldots,\sigma_r$ (if $r\ge 1$) and its complex embeddings are $\tau_1,\ldots,\tau_s, \overline{\tau_1},\dots,\overline{\tau_s}$ (if $s \ge 1$) 
	which naturally yields the vector embedding 
	$$
	{\psi \, \colon \, \K
		\hookrightarrow \R^{r} \times \C^{s} \cong \R^{n}},
	$$ 
	where 
	$$\psi(\alpha)=(\sigma_1(\alpha),\dots,\sigma_r(\alpha),\Re(\tau_1(\alpha)),\Im(\tau_1(\alpha)),\dots,\Re(\tau_s(\alpha)),\Im(\tau_s(\alpha))).$$
	Hence, we can consider the  full rank $n$-dimensional lattice $\Lambda _{\K} =\psi (\Z_{\K})$, 
	see~\cite[Chapter~8, Section~7]{ConSlo} or~\cite[Chapter~V]{Lang}, where also many 
	applications of this construction are given.

	It  is well known that  the determinant (that is, the invariant of a lattice defined as
	the volume of the parallelepiped formed by any basis of this lattice)   of  $\Lambda _{\K}$
	is given by 
	\begin{equation}
		\label{eq:disc det}
		\det (\Lambda _{\K}) = 2^{-s}D_{\K}^{1/2}
	\end{equation}
	(see~\cite[Chapter~V, Section~2, Lemma~2]{Lang}). 
	
	We also need some facts and notions from the theory of geometric lattices.
	Recall that a subset of $\R^n$ is called a lattice
	if it is a discrete subgroup of $\R^n$.
	In the sequel, let $\Lambda $ be a lattice of $\R^n$.
	
	Recall that the \textit{length} of a vector $\vec{x}=(x_1, \ldots, x_n) \in \R^n$ is defined by 
	$$
	\| \vec{x} \| = \left(x_1^2 + \cdots + x_n^2 \right)^{1/2}. 
	$$
	
	Next, we need the standard notion of the  \textit{successive minima} of a lattice $\Lambda$ in $\R^n$ (see~\cite{ConSlo, Mart}).
	For $k =1, \ldots, n$, we recall that $\mu_k(\Lambda)$ is the smallest real number $\mu$
	such that $\Lambda$  contains $k$ linearly independent vectors of length at most $\mu$. 
	
	We denote by $\det(\Lambda)$ the determinant of $\Lambda$, 
	which is the volume of the parallelepiped formed by any basis of $\Lambda$.
	
	Recall that the \textit{Hermite constant} for dimension $n$ is 
	$$
	\gamma_n := \sup_{\Lambda}\mu_1(\Lambda)^2 \det(\Lambda)^{-2/n},
	$$
	where $\Lambda$ runs over all lattices in $\R^n$.
	
	We remark that the successive minima defined in~\cite[Definition~2.6.7]{Mart} are the squares of the ones defined above, and  the determinant $\det(\Lambda)$ defined in~\cite[Definition~1.2.4]{Mart} is the square of the one defined here. 
	
	Now, we can restate Minkowski's theorem as follows (see~\cite[Theorem~2.6.8]{Mart}). 
	
	\begin{lemma}    \label{lem:Mink}
		For any lattice $\Lambda \in \R^n$ and any positive integer $k \le n$, we have 
		$$
		\mu_1(\Lambda) \cdots \mu_k(\Lambda) \le \gamma_n^{k/2} \det(\Lambda)^{k/n}. 
		$$
	\end{lemma}

	For a nonnegative integer $s\le n/2$ and $r = n-2s$ 
	we define the function 
	\begin{equation}
		\label{boi}
		N_s(\vec{u}) := \prod_{i =1}^{r} |u_i| \cdot  \prod_{j=1}^{s} (v_j^2 + w_j^2)
	\end{equation}
	on vectors
	$\vec{u}= (u_{1}, \ldots\,, u_{r}, v_{1},  w_{1}, \ldots\,,
	v_{s}, w_{s}) \in  \R^n$. 
	Set 
	$$
	m_s(\Lambda) := \inf_{\vec{x}\in \Lambda \setminus\{0\}} 
	N_s(\vec{x})
	$$
	and
	\begin{equation}\label{vienas1}
		M_s(\Lambda) := \sup_{\vec{u} \in \R^n} \inf_{\vec{x}\in \Lambda} 
		N_s(\vec{u} - \vec{x}), 
	\end{equation}
	which are called the \textit{homogeneous minimum} and the \textit{inhomogeneous minimum} of $\Lambda$ respectively. 
	
	The following lemma follows closely the proof of~\cite[Lemma~2.1]{DSS}. 
	We remark that the pair of parameters $(r,s)$  in our notation corresponds to $(s,t)$ in the notation of~\cite{DSS}.
	
	\begin{lemma} \label{lem:vecx}
		For any nonnegative integer $s \le n/2$ and any nonzero vector $\vec{x} \in \R^n$, we have 
		$$
		\|\vec{x}\|   \ge 2^{-s/n} n^{1/2} N_s({\vec x})^{1/n}.
		$$
	\end{lemma}
	
	\begin{proof}
		Set $r = n-2s$, and write the nonzero vector $\vec{x} \in \R^n$ as
		$$
		\vec{x}= (x_{1}, \ldots\,, x_{r}, y_{1}, z_{1}, \ldots\,,y_{s}, z_{s}).
		$$  
		Then,
		$$
		\|\vec{x}\|^2= \sum^{r}_{i=1} x^{2}_{i} +
		\sum^{s}_{j=1} (y^{2}_{j} + z^{2}_{j}). 
		$$
		Now, let us write 
		$$
		N_s(\vec{x}) = XY,
		$$
		where 
		$$
		X:=  \prod_{i =1}^{r} |x_i| \mand Y:=  \prod_{j =1}^{s} (|y_j|^2 + 
		|z_j|^2).
		$$
		
		Assume that $XY \ne 0$, since otherwise the result is trivial.
		For $rs> 0$,  by   the inequality  between the arithmetic and geometric means and~\eqref{boi}, 
		we obtain 
		$$
		\|\vec{x}\|^2 \geq r X^{2/r}+ sY^{1/s} = rX^{2/r}+ sN_s(\vec{x}) ^{1/s}X^{-1/s}.
		$$
		For $A>0$, the minimum of the function
		$F(x)=rx^{2/r}+sA^{1/s} x^{-1/s}$  on the half line $x > 0$ is attained for 
		$x_0:= 2^{-rs/(r+2s)} A^{r/(r+2s)}$, 
		which is the root of the equation $F'(x)=2x^{2/r-1} - A^{1/s} x^{-1/s-1}=0$. By $n=r+2s$,
		the minimum is equal to 
		\begin{align*}
			F(x_0) &=rx_0^{2/r}+sA^{1/s}x_0^{-1/s} =  (r+2s) x_0^{2/r} \\
			& = (r+2s) 2^{-2s/(r+2s)} A^{2/(r+2s)}   = n2^{-2s/n} A^{2/n}. 
		\end{align*}
		Now,
		using this inequality with $A = N_s({\vec x})$,
		we derive that
		\begin{equation}
			\label{eq:Bound st>0}
			\|\vec{x}\|^2   \ge n2^{-2s/n} N_s({\vec x})^{2/n}.
		\end{equation}
		
		One also verifies that for $r = 0$  we have $s=n/2$, $Y=N_s(\vec{x})$ and thus 
		$$
		\|\vec{x}\|^2 \geq sY^{1/s} =  sN_s(\vec{x}) ^{1/s}=\frac{n}{2} N_s({\vec x})^{2/n}; 
		$$
		while for $s = 0$ we have  $r=n$, $X=N_s(\vec{x})$ and hence
		$$
		\|\vec{x}\|^2 \geq r X^{2/r}  =  r N_s(\vec{x})^{2/r}= n N_s({\vec x})^{2/n}.
		$$
		Therefore, the inequality~\eqref{eq:Bound st>0} also holds for $rs = 0$. 
		
		Finally, from~\eqref{eq:Bound st>0}, by taking the square roots
		of both sides, we get the desired inequality about the length $\|\vec{x}\|$. 
	\end{proof}

	By Lemma~\ref{lem:vecx}, we get the following improvement of~\cite[Lemma~3.4]{B-FBJ}. More precisely, instead of the factor $2^{n/2}$ in~\cite[Lemma~3.4]{B-FBJ}, it contains the factor $2^s\le 2^{n/2}$
	with a strict inequality for $r > 0$.  
	
	\begin{lemma}
		\label{lem:m and mu1} 
		For any nonnegative integer $s \le n/2$, we have 
		$$
		m_s(\Lambda) \le 2^{s} n^{-n/2} \mu_1(\Lambda)^n.
		$$
	\end{lemma}
	
	\begin{proof}
		Note that the inequality in Lemma~\ref{lem:vecx} holds for any vector $\vec{x}\in \Lambda \setminus\{0\}$. 
		This, together with the definitions of $\mu_1(\Lambda)$ and $m_s(\Lambda)$, yields
		$$
		\mu_1(\Lambda) \ge 2^{-s/n} n^{1/2} m_s(\Lambda)^{1/n},
		$$
		which implies the assertion of the lemma. 
	\end{proof}
	
	Using Lemma~\ref{lem:vecx},  we also obtain the following improvement of~\cite[Lemma~3.5]{B-FBJ}.  
	More precisely, instead of the factor $2^{-n/2}$ as  in~\cite[Lemma~3.5]{B-FBJ}, 
	it contains the factor $2^{s-n} \le 2^{-n/2}$ with a strict inequality for $r > 0$. 
	
	\begin{lemma}
		\label{lem:M and mun} 
		For any nonnegative integer $s \le n/2$, we have 
		$$
		M_s(\Lambda) \le 2^{s-n} \mu_n(\Lambda)^n.
		$$
	\end{lemma}
	
	\begin{proof}
		It is shown in~\cite[Lemma~3.5]{B-FBJ}  that for any 
		${\vec u} \in \R^n$ there is a vector ${\vec x} \in \Lambda$ such that
		$$
		\|{\vec u}-{\vec x}\| \le \frac{\sqrt{n}}{2} \mu_n(\Lambda).
		$$  
		Now, using Lemma~\ref{lem:vecx} (with ${\vec x}$ replaced by ${\vec u} -{\vec x}$) we find that 
		\begin{align*}
			N_s({\vec u}-{\vec x}) \le 2^s n^{-n/2} \|\vec{u} - \vec{x}\|^n 
			& \le 2^s n^{-n/2} \cdot \left( \frac{\sqrt{n}}{2} \mu_n(\Lambda) \right)^n \\
			& = 2^{s-n} \mu_n(\Lambda)^n. 
		\end{align*}
		Finally, since the above vector $\vec{u}$ is arbitrary, the desired result follows by~\eqref{vienas1}. 
	\end{proof}
	
	Now, we restate some constructions from~\cite[Section~5]{B-FBJ}. 
	(We remark that the definition of the group $G$ in~\cite[Section~5]{B-FBJ} has been corrected in its latest 2023 version in arXiv.)
	In the sequel, we use $G_{r+s}$ to denote this group. 
	
	Using the isomorphism $\R^n \cong \R^r \oplus \C^s$, we can consider  $\Lambda$ 
	as a lattice in  $\R^r \oplus \C^s$. 
	Let $G_{r+s}$ be the group of diagonal matrices 
	$g = \diag(g_1, \ldots, g_{r+s})$ with positive diagonal 
	entries $g_1, \ldots, g_{r+s}$ such that 
	$$
	g_1 \cdots g_s (g_{s+1}\cdots g_{r+s})^2 = 1. 
	$$ 
	We can consider the natural action of the elements 
	of the group $G_{r+s}$ on the elements of $\Lambda$.  
	
	We recall that~\cite[Theorem~5.1]{B-FBJ} asserts that if $G_{r+s}\Lambda$ is compact, 
	then for any positive integer $a \le r+s$ we have 
	\begin{equation} \label{eq:mM}
		m_s(\Lambda)^s M_s(\Lambda)^a \le \left(2^{s-a} n^{-s}  \gamma_n^{s+a} \right)^{n/2} \det(\Lambda)^{s+a}, 
	\end{equation}
	which can be equivalently written as 
	$$
	m_s(\Lambda)^s M_s(\Lambda)^a \le 2^{s^2+as-an+r(a+s)/2} n^{-ns/2}  \gamma_n^{n(s+a)/2} \det(\Lambda)^{s+a},
	$$
	since $n = r+2s$.
	We refer to~\cite[Section~4.5]{B-FBJ} for the definition of compactness.
	
	We now present the following improvement of~\eqref{eq:mM},
	which rests on an idea of McMullen~\cite{McM}, which we borrow without any changes from 
	the argument of~\cite{B-FBJ} combined with Lemmas~\ref{lem:m and mu1} and~\ref{lem:M and mun}.
	
	\begin{lemma}
		\label{lem:m and M} Assume that the set $G_{r+s}\Lambda$ is compact.
		Then, for any positive integer $a \le r+s$, we have
		$$
		m_s(\Lambda)^s M_s(\Lambda)^a \le 2^{s^2+as-an} n^{-ns/2}  \gamma_n^{n(s+a)/2} \det(\Lambda)^{s+a}. 
		$$
	\end{lemma}
	
	\begin{proof}
		Since $G_{r+s}\Lambda$ is compact, using~\cite[Theorem~4.3]{B-FBJ}, there exists $g \in G_{r+s}$ for which
		$$
		\mu_{s+1}(g\Lambda)=\dots=\mu_n(g \Lambda).
		$$ 
		By Minkowski's theorem (Lemma~\ref{lem:Mink}), for any positive integer $k \leq n$, we have 
		$$
		\mu_1(g\Lambda) \cdots \mu_{k}(g \Lambda) \le \gamma_n^{k/2} \det(g \Lambda)^{k/n}.
		$$
		Therefore, selecting $k=s+a$, we find that 
		\begin{align*}
			\gamma_n^{(s+a)/2} \det(g \Lambda)^{(s+a)/n} &\ge
			\mu_1(g \Lambda) \cdots \mu_{s+a}(g \Lambda) \\&= \mu_1(g \Lambda)
			\cdots \mu_{s}(g \Lambda) \mu_n(g \Lambda)^a   \ge
			\mu_1(g \Lambda)^s \mu_n(g \Lambda)^a.
		\end{align*}
		In addition, by Lemma~\ref{lem:m and mu1}, we get
		$$
		\mu_1(g \Lambda) \geq 2^{-s/n} n^{1/2} m_s(g \Lambda)^{1/n} = 2^{-s/n} n^{1/2} m_s(\Lambda)^{1/n},
		$$
		while, by Lemma~\ref{lem:M and mun}, we find that 
		$$
		\mu_n(g \Lambda) \geq 2^{1-s/n} M_s(g \Lambda)^{1/n}=  
		2^{1-s/n} M_s(\Lambda)^{1/n}.
		$$
		Finally, since $\det(g\Lambda)=\det(\Lambda)$, combining the above estimates we obtain
		\begin{align*}
			\gamma_n^{(s+a)/2} \det(\Lambda)^{(s+a)/n} 
			& \ge 2^{-s^2/n} n^{s/2} m_s(\Lambda)^{s/n} \cdot 2^{a-as/n} M_s(\Lambda)^{a/n} \\
			& = 2^{(an-s^2-as)/n} n^{s/2} m_s(\Lambda)^{s/n} M_s(\Lambda)^{a/n}.
		\end{align*}
		This yields the desired inequality of the lemma. 
	\end{proof}

	\section{Proof of Theorem~\ref{thm: MKa}}
	
	Noticing the embedding $\psi \, \colon \, \K
	\hookrightarrow \R^{r} \times \C^{s} \cong \R^{n}$, we directly have 
	\begin{equation}  \label{eq:MMs}
		\rM(\K) \le M_s(\Lambda_\K).
	\end{equation}
	So, it suffices to prove the desired upper bound for $M_s(\Lambda_\K)$. 
	
	It is clear that $m_s(\Lambda_{\K})=1$. 
	We also note that the set $G_{r+s}\Lambda_\K$ is compact (see~\cite[Section~5.2]{B-FBJ} and its latest arXiv version for a correct proof). 
	Then, Lemma~\ref{lem:m and M} gives
	\begin{equation}\label{boi1}
		M_s(\Lambda_\K)^a \leq 2^{s^2 + as -an} n^{-ns/2}  \gamma_n^{n(s+a)/2} \det(\Lambda_\K)^{s+a}.
	\end{equation}
	By~\eqref{eq:disc det}, we get 
	$$
	\det (\Lambda_{\K})^{s+a}=2^{-s^2-as}D_{\K}^{(s+a)/2}. 
	$$
	Inserting this expression into \eqref{boi1} we obtain 
	$$
	M_s(\Lambda_\K)^a \leq 2^{-an} n^{-ns/2}  \gamma_n^{n(s+a)/2} D_{\K}^{(s+a)/2},
	$$
	and hence
	$$
	M_s(\Lambda_\K) \leq 2^{-n} n^{\frac{-ns}{2a}}  \gamma_n^{\frac{n(s+a)}{2a}} D_{\K}^{\frac{s+a}{2a}}. 
	$$
	This, together with~\eqref{eq:MMs}, yields the desired result. 
	
	\section{Some numerical calculations}
	\label{sec:calc}
	
	Now, we want to illustrate some improvements of our previous results given by Theorem~\ref{thm: MKa}. 
	
	The exact values of the Hermite constants $\gamma_n$ are known only for $1 \le n \le 8$, see~\cite[Table~14.4.1]{Mart},  and also for $n=24$, see~\cite[Theorem~9.3]{CK}. 
	(There is also a conjecture related to the cases $9 \le n \le 23$, see~\cite{Ma}.)
	In Table~\ref{tab:Hermite} we summarize all the cases established so far. 
	
	\begin{table}  [H]
		\begin{center}
			\begin{tabular}{|c|c|c|c|c|c|c|c|c|c|}
				\hline
				$n$ & $1$ & $2$ & $3$ & $4$ & $5$ & $6$ & $7$ & $8$ & $24$   \\ \hline
				
				$\gamma_n$ & $1$ & $2 \cdot 3^{-1/2}$ & $2^{1/3}$ & $2^{1/2}$ & $2^{3/5}$ & $2 \cdot 3^{-1/6}$ 
				& $2^{6/7}$ & $2$ & $4$ \\ \hline
			\end{tabular}
			\caption{Hermite constants}
			\label{tab:Hermite}
		\end{center}
	\end{table}

	Combining Theorem~\ref{thm: MKa} with Table~\ref{tab:Hermite}, we can obtain some improvements upon previous upper bounds 
	for $\rM(\K)$ when $n$ is small. 
	For example, comparing with the table in~\cite[Section~5.4]{B-FBJ}, we 
	present Table~\ref{tab:MK} with better upper bounds on  $\rM(\K)$.
	
	\begin{table} 
		\begin{center}
			\begin{tabular}{|c|c|c|c|}
				\hline
				$n$ & $s$ & $a$ & Upper bound for $\rM(\K)$  \\ \hline
				
				$1$ & $0$ & 0   & $2^{-1} D_{\K}^{1/2}=0.5 \cdot D_{\K}^{1/2}$  \\ \hline 
				
				$2$ & $0$ & 0 & $2^{-1} 3^{-1/2}D_{\K}^{1/2}=0.28867 \ldots \cdot D_{\K}^{1/2}$  \\ \hline
				
				$2$ & $1$ & 0 & $6^{-1}D_{\K}=0.16666\ldots \cdot D_{\K}$  \\ \hline
				
				$3$ & $0$ & $1,2$ or $3$ & $2^{-5/2}D_{\K}^{1/2}=0.17677 \ldots \cdot D_{\K}^{1/2}$  \\ \hline

				$3$ & $1$ & $1$ & $2^{-2}3^{-3/2}D_{\K}=0.04811\ldots \cdot D_{\K}$  \\ \hline
				
				$3$ & $1$ & $2$ & $2^{-9/4}3^{-3/4} D_{\K}^{3/4} =0.09222 \ldots \cdot D_{\K}^{3/4}$  \\ \hline
				
				$4$ & $0$ & $1,2,3$ or $4$ & $2^{-3}D_{\K}^{1/2}=0.125\cdot D_{\K}^{1/2}$  \\ \hline
				
				$4$ & $1$ & $1$ & $2^{-6}D_{\K} =0.015625 \cdot D_{\K}$  \\ \hline
				
				$4$ & $1$ & $2$ & $2^{-9/2}D_{\K}^{3/4}= 0.04419\ldots \cdot D_{\K}^{3/4}$  \\ \hline
				
				$4$ & $1$ & $3$ & $2^{-4}D_{\K}^{2/3}=0.0625\cdot D_{\K}^{2/3}$  \\ \hline
				
				$4$ & $2$ & $1$ & $2^{-9} D_{\K}^{3/2} = 0.001953125 \cdot D_{\K}^{3/2}$  \\ \hline
				
				$4$ & $2$ & $2$ & $2^{-6}D_{\K} = 0.015625\cdot D_{\K}$  \\ \hline
				
				$5$ & $0$ & $1,2,3,4$ or $5$ & $2^{-7/2}D_{\K}^{1/2}=0.08838\ldots \cdot D_{\K}^{1/2}$  \\ \hline
				
				$5$ & $1$ & $1$ & $2^{-2}5^{-5/2}D_{\K} = 0.00447\ldots \cdot D_{\K}$  \\ \hline
				
				$5$ & $1$ & $2$ & $2^{-11/4} 5^{-5/4}D_{\K}^{3/4}=0.01988\ldots \cdot D_{\K}^{3/4}$  \\ \hline
				
				$5$ & $1$ & $3$ & $2^{-3} 5^{-5/6}D_{\K}^{2/3}=0.03269\ldots \cdot D_{\K}^{2/3}$  \\ \hline
				
				$5$ & $1$ & $4$ & $2^{-25/8} 5^{-5/8}D_{\K}^{5/8} = 0.04192 \ldots \cdot D_{\K}^{5/8}$  \\ \hline
				
				$5$ & $2$ & $1$ & $2^{-1/2}5^{-5}D_{\K}^{3/2} =0.00022\ldots \cdot D_{\K}^{3/2}$  \\ \hline
				
				$5$ & $2$ & $2$ & $2^{-2}5^{-5/2}D_{\K}= 0.00447\ldots \cdot D_{\K}$  \\ \hline
				
				$5$ & $2$ & $3$ & $2^{-5/2}5^{-5/3}D_{\K}^{5/6}= 0.01209\ldots \cdot D_{\K}^{5/6}$  \\ \hline
			\end{tabular}
			\vspace{3mm}
			\caption{Upper bound for $\rM(\K)$ ($1 \le n \le 5$)}
			\label{tab:MK}
		\end{center}
	\end{table}  
	
	So far, for an arbitrary number field $\K$ of degree $n$ (without any conditions on its signature), 
	the best upper bound for $\rM(\K)$ is 
	\begin{equation} \label{eq:B-F}
		\rM(\K) \le 2^{-n} D_\K,
	\end{equation}
	which has been proved by Bayer-Fluckiger~\cite{B-F}. 
	
	For a number field $\K$ of degree $n$ and of signature $(r,s)$, 
	if $s \ge 1$, using Theorem~\ref{thm: MKa} and letting $a=s$ we obtain 
	\begin{equation}   \label{eq:a=s}
		\rM(\K) \le 2^{-n}n^{-n/2}\gamma_n^n D_\K, 
	\end{equation}
	which is better than the upper bound in~\eqref{eq:B-F} when $\gamma_n < \sqrt{n}$. 
	
	Recall that currently the strongest upper bound for $\gamma_n$ is 
	\begin{equation}  \label{eq:Blich}
		\gamma_n \le \frac{2}{\pi} \cdot \Gamma(2+n/2)^{2/n}, 
	\end{equation}
	where $\Gamma$ is the gamma function, is due to Blichfeldt~\cite{Blich}. 
	
	Now, by direct computation (for example, using Pari/Gp), we find that the right-hand side of~\eqref{eq:Blich} is less than $\sqrt{n}$ 
	for $n$ in the range $2 \le n \le 43$. Thus, the bound in~\eqref{eq:a=s} is better than the bound in~\eqref{eq:B-F} for that $n$.
	
	\section*{Acknowledgements} 
	
	The authors would like to thank Eva Bayer-Fluckiger, Martino Bo\-rello and  Peter Jossen for valuable discussions and 
	especially for updating their paper~\cite{B-FBJ}.  
	
	During the preparation of this  work, 
	M.~Sha was supported in part by the Guangdong Basic and Applied Basic Research Foundation, Grant 2022A1515012032, 
	and 
	I.~E.~Shparlinski   by the Australian Research Council, Grants DP230100530 and  DP230100534.

\end{document}